\documentclass[12pt,sans]{iopart}

\usepackage[utf8]{inputenc}

\usepackage{iopams}
\usepackage{amstext,ifthen,fancyhdr,lastpage,mathrsfs,color,listings,dsfont,amsthm}
\usepackage{graphicx}
\usepackage{xcolor}

\usepackage{hyperref}
\hypersetup{
    colorlinks,
    linkcolor={blue!100!black},
    citecolor={blue!100!black},
    urlcolor={blue!100!black}
}

\newcommand{\mN}{\mathbb{N}}
\newcommand{\mZ}{\mathbb{Z}}

\newcommand{\mR}{\mathbb{R}}
\newcommand{\mC}{\mathbb{C}}

\newcommand{\Text}[1]{\text{\textnormal{#1}}}

\newcommand{\I}{\Text{i}}
\newcommand{\E}{\Text{e}}

\newcommand{\MTEXT}[1]{\;\;\;\;\;\text{#1}\;\;\;\;\;}
\newcommand{\D}{\text{d}}

\newcommand{\CpR}[2]{\mathscr{C}^{#1}(\mathbb{R}^{#2})}

\newcommand{\conj}{\overline}

\newcommand{\Sdash}[1]{\mathscr{S}'(\mR^{#1})}
\newcommand{\Sdashc}[1]{\mathscr{S}_{\Text{c}}'(\mR^{#1})}
\newcommand{\DF}{\mathcal{D}_d^{\Text{F}}}
\newcommand{\tDF}{\mathcal{ D}^{\Text{F}}}
\newcommand{\FF}{\mathcal{F}}
\newcommand{\RR}{\mathcal{R}}

\newcommand{\wF}{w_{\Text{F}}}

\newcommand{\adj}[1]{#1^\ast}
\newcommand{\mSn}[1]{\mathbb{S}^{#1}}

\newcommand{\thmref}[1]{Theorem~\ref{#1}}

\newcommand{\lemref}[1]{Lemma~\ref{#1}}

\newcommand{\cref}[1]{Corollary~\ref{#1}}
\newcommand{\Cref}[1]{Corollary~\ref{#1}}

\newtheorem{theorem}{Theorem}
\newtheorem{lemma}[theorem]{Lemma}
\newtheorem{cor}[theorem]{Corollary}

\begin{document}

\title[Uniqueness of propagation-based phase contrast imaging]{A uniqueness result for propagation-based phase contrast imaging from a single measurement}

\author{Simon Maretzke  $^1$$^2$}

\address{$^1$ Institute for X-Ray Physics, University of G\"ottingen, Friedrich-Hund-Platz 1, D-37077 G\"ottingen, Germany \\
 $^2$ Institute for Numerical and Applied Mathematics, University of G\"ottingen, Lotzestr 16-18, D-37083 G\"ottingen, Germany}

\ead{simon.maretzke@googlemail.com}
\vspace{10pt}

\begin{abstract}
Phase contrast imaging seeks to reconstruct the complex refractive index of an unknown sample from scattering intensities, measured for example under illumination with coherent X-rays. By incorporating refraction, 
this method yields improved contrast compared to purely absorption-based radiography  but involves a phase retrieval problem which, in general, allows for ambiguous reconstructions. In this paper, we show uniqueness of propagation-based phase contrast imaging for compactly supported objects in the near-field regime, based on a description by the projection- and paraxial approximations. In this setting, propagation is governed by the Fresnel propagator and the unscattered part of the illumination function provides a known reference wave at the detector which facilitates phase reconstruction. The uniqueness theorem is derived using the theory of entire functions. Unlike previous results based on exact solution formulae, it is valid for arbitrary complex objects and requires intensity measurements only at a single detector distance and illumination wavelength. We also deduce a uniqueness criterion for phase contrast tomography, which may be applied to resolve the three-dimensional structure of micro- and nano-scale samples. Moreover, our results may have some significance to electronic imaging methods due to the equivalence of paraxial wave propagation and Schr\"odinger's equation.
\end{abstract}


\vspace{2pc}

\submitto{\IP}
%
\maketitle
%
%

\section{Introduction} \label{S1}

The advent of coherent X-ray sources, such as synchrotrons and - more recently - free-electron-lasers, has allowed to extend the scope of radiography to quasi-transparent specimen by phase-sensitive imaging techniques \cite{Nugent2010coherent}. Examples include micro- or nano-scale objects composed mainly of light elements, most prominently biological cells \cite{Bartels2012,Miao2003EColiFarfield,Thibault2006YeastCell} but also organic or ceramic foams \cite{Barty2008ceramicfoam,Cloetens1999}. Phase contrast imaging seeks to reconstruct the spatially varying complex refractive index $n = 1 - \delta + \I \beta$ of such samples from measurements of scattered wave fields. In particular, the approach takes into account the real component $\delta$ governing the refractive phase shifts that are imprinted to transmitted radiation. For the specimens and wavelengths in question, $\delta$ is typically up to three orders of magnitude larger than the absorptive part $\beta$ \cite{Cloetens1999,Henke1993TypicalDeltaBeta,Mayo2002quantitative}. Consequently, solely absorption-based approaches are bound to result in poor contrast. 
On the other hand, refractive information is encoded entirely in the phase of the transmitted wave fields which cannot be observed directly by common CCD detectors due to their physical limitation to wave \emph{intensities}  \cite{PaganinXRay}. The required phase sensitivity can be achieved either by interferometric techniques \cite{Bonse1965interferometric,Momose1995interferometric,Wilkins1996interferometric}, or by measuring the \emph{propagated} wave field on a distant detector rather than close to the exit-surface of the sample \cite{Gureyev1999noninterferometric,Nugent1996Propagationbased,Paganin1998noninterferometric,Pogany1997noninterferometric}. In the latter setting, which is studied in this work, diffraction encodes the phase information into observable intensities. 
This naturally raises the question whether the encoding is unambiguous, i.e.\ whether the phase can be recovered uniquely from the data.

In the \emph{far-field} limit of large distances between sample and detector where the propagation essentially reduces to a Fourier transform \cite{PaganinXRay}, this problem has been subject to extensive analytical studies based on the complex analysis approach of Akutowicz and Walther \cite{Akutowicz1956I,Akutowicz1957II,Walther1963}. See for instance \cite{Klibanov1995,Millane1990} for reviews. The principal result for the reconstruction of a compactly supported function from Fourier intensity data is that solutions to the phase retrieval problem are in general highly non-unique in $\mR$. {On the contrary, non-trivial ambiguities are ``pathologically rare'' in higher dimensions \cite{Barakat1984,Fienup19782DPhaseRetrFeasible}, occurring only for objects within a set of measure zero \cite{BruckSodin1979PhaseAmbiguity2D,Hayes1982Reducible}. {Absolute uniqueness can be shown to hold under additional a priori assumptions on the object, such as vanishing absorption, suitable regularity and symmetric- or symmetry-breaking structure \cite{Klibanov20062DUniquePurePhase,Klibanov20141Dunique}.} 
However, \emph{ab initio} reconstructions require iteratively updated support estimates \cite{Fienup1986algorithm,Marchesini2003PhysRevB} in order to overcome the ``trivial'' ambiguities induced by the invariance of the data under translations and reflections of the object.} Far-field phase contrast, better known as coherent diffractive imaging, has been successfully applied to 2D- and 3D-imaging of quasi-transparent specimen ($\beta = 0$) \cite{MiaoNature1999CDIFirstExp,Miao2003EColiFarfield} and single-material objects ($\beta \propto \delta$) \cite{Chapman2006,Marchesini2003GoldSpFarfield}.

On the contrary, this work is concerned with phase contrast imaging in the \emph{near-field regime}, also called the \emph{holographic-} or \emph{Fresnel regime}, characterized by moderate detector distances. Propagation in this regime is described by the Fresnel propagator. The detected wave field is composed of the incident illumination beam plus a perturbation induced by the scattering on the object. In the far-field case, the former component usually gives non-negligible contributions only around the center of the diffraction pattern, often blocked by a beam stop in order not to damage the detector. In the near-field regime, in contrast, the unscattered part of the probing beam typically manifests itself in a bulk background intensity at the detector, representing a \emph{global} reference wave.
As discussed in \cite{Millane1990} for the example of speckle holography, the presence of such a known reference signal in the data may eliminate phase retrieval ambiguities.
 
{The near-field phase retrieval problem has been proven to be uniquely solvable for general compactly supported objects, given at least \emph{two} independent intensity patterns recorded at different detector distances or incident wavelengths \cite{Jonas2004TwoMeasUniquePhaseRetr}. To the best of our knowledge, no equally general analogue has ever been derived for a \emph{single} measurement setting.} Previous results \cite{Klibanov1986FresnelUnique} only guarantee unique recovery of either the phase or the amplitude of a complex-valued signal if the other part is known. However, the latter study does not exploit the perturbational algebraic structure of near-field data arising from the superposition of the unknown object with the reference probe wave field.

Exact solution formulae for near-field phase contrast imaging from a single intensity measurement are restricted to single-material samples. Additionally, these approaches assume small propagation distances to approximate the transport-of-intensity equation \cite{Paganin2002simultaneous,Teague1983TIE} or weak absorption and slowly varying phase shifts \cite{Turner2004FormulaWeakAbsSlowlyVarPhase}. For general objects, $\delta$ and $\beta$ have to be determined independently.
Referring to the ``phase vortex'' counter-example \cite{Nugent2007TwoPlanesPhaseVortex}, it is commonly argued that two real-valued intensity patterns are not only sufficient but also necessary for uniqueness of such a reconstruction \cite{Burvall2011TwoPlanes}. Yet, we emphasize that the studied vortical wave fields  may never arise from scattering on compact samples. Moreover, recent numerical results \cite{Ruhlandt2014} for phase contrast tomography suggest that intensity data from a single detector distance may be sufficient for unique recovery of $\delta$ and $\beta$.
This work indeed aims to disprove the widely believed existence of ambiguities in single-distance near-field phase contrast imaging - at least for compactly supported objects illuminated by plane waves or Gaussian beams.
Our central uniqueness result, based on growth estimates for entire functions, reads as follows:
\vspace{1em}\begin{theorem} \label{thm:1}
 Let $\Sdash{m} \supset \Sdashc{m}$ denote the tempered (and compactly supported) distributions and $\FF: \Sdash{m} \to \Sdash{m}$ the Fourier transform. For $w \in \CpR{\infty}{m}$ everywhere nonzero, $\alpha \in \mC \setminus \mR$ and $\check p \in \Sdashc{m} \setminus \{ 0 \}$ define forward operators $F, F_{\Text{lin}}: \Sdashc{m} \to \CpR{\infty}{m}$ by
  \numparts \label{eq:defOps}
 \begin{eqnarray}
  &F(h) &= | \FF(\check p )\exp(\alpha ( \bi \cdot )^2) + \FF(w \cdot h) |^2 \label{eq:defF} \\ 
  &F_{\Text{lin}}(h) &= | \FF(\check p )\exp(\alpha ( \bi \cdot )^2) + \FF(w \cdot h) |^2 - | \FF(w \cdot h) |^2 \label{eq:defFlin} 
 \end{eqnarray}
 \endnumparts
 Then $F$ and $F_{\Text{lin}}$ are well-defined and injective. Moreover, any $ h\in \Sdashc{m}$ is uniquely determined by data $F(h)_{|U}$ or $F_{\Text{lin}}(h)_{|U}$ restricted to an arbitrary open set $U \subset \mR^m$.
\end{theorem}\vspace{1em}
{Physically, the first summand in \eref{eq:defF} is associated with the unscattered probe beam, whereas the second gives the scattering perturbation $h$ induced by the specimen. The operator $F_{\Text{lin}}$ represents a linearization of $F$ valid for small $h$, i.e.\ for weak objects in a suitable sense, an assumption which is underlying to commonly used near-field phase retrieval techniques based on the so-called contrast transfer function \cite{Guigay1977CTF,Pogany1997noninterferometric,Cloetens1999,Mayo2002quantitative,Krenkel2014BCAandCTF}. Details on the physical problem of near-field phase contrast imaging and how it matches the framework of \thmref{thm:1} are discussed in \sref{S2}.} \Sref{S3} gives a recap of the theory of entire functions as a preparation for the proofs of the main results in \sref{S4}.

\section{Physical problem} \label{S2}

\Fref{figure1} shows an idealized setup for propagation-based phase contrast imaging with X-rays: incident coherent electromagnetic waves of wavenumber $k$ interact with an unknown object of thickness $L$ in the beam line, leading to a slightly perturbed wave field at the exit-surface $E_0$. We parametrize the object by its refractive index $n = 1- \delta + \I \beta$ where $\delta$ and $\beta$, governing refraction are absorption, respectively, are real-valued and compactly supported. The intensity of the scattered radiation is measured in the detector plane $E_d$ at finite distance $d>0$ to the object. Although the physical setting in \Fref{figure1} is three-dimensional, we consider the more general case of $m \in \mN$ lateral dimensions, denoted by $\bi x$, plus the axial $z$-direction. 
 	\begin{figure}[hbt!]
 	 \centering
 	 \includegraphics[width=.7\textwidth]{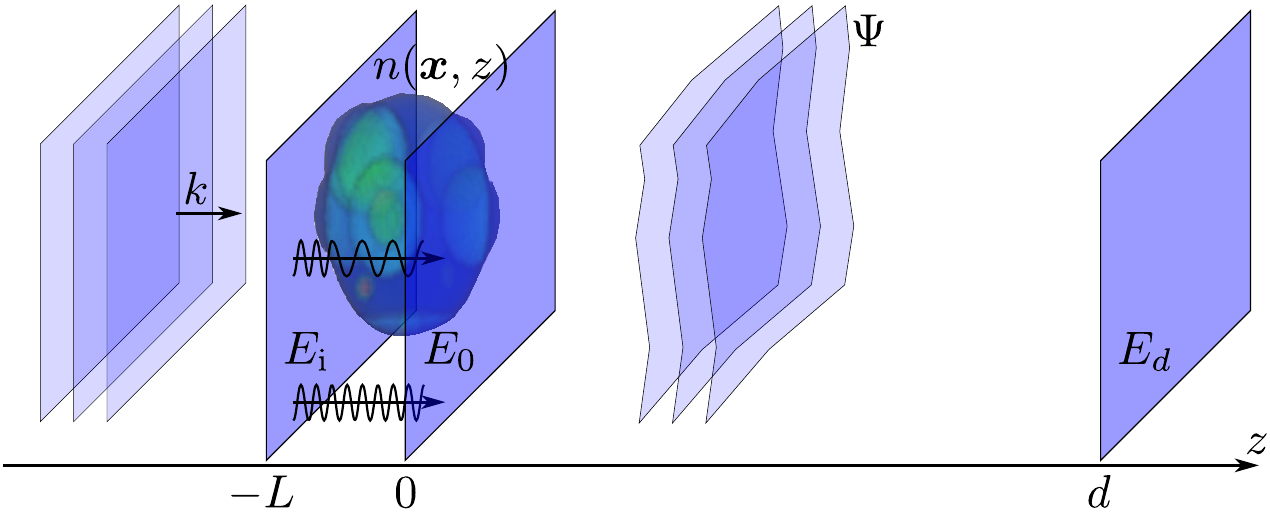}
 	 \caption{Idealized setup for propagation-based phase contrast imaging: incident coherent radiation, visualized by plane wave fronts, is scattered on a compactly supported object. The resulting phase shifts and absorption manifest in the intensity profiles recorded at $E_d$ (courtesy of Aike Ruhlandt, personal communication). \label{figure1}}
 	\end{figure}
 	
{It is well known that the cartesian components of a monochromatic electromagnetic wave in a medium of refractive index $n$ can be described by a single complex-valued time-independent field $\Psi$, governed by the \emph{Helmholtz equation} \cite{PaganinXRay}
\begin{equation}
  \Delta \Psi + k^2 n^2 \Psi=0. \label{eq:Helmholtz}
\end{equation}
The model of phase contrast imaging considered in this work is based on the \emph{paraxial-} and the \emph{projection approximations} \cite{PaganinXRay}. We discuss these briefly here, referring to \cite{Jonas2004TwoMeasUniquePhaseRetr} for detailed error estimates.

The paraxial approximation requires that $\Psi$ is of the form $\Psi(\bi x, z) = \E^{\I k z} \tilde \Psi(\bi x, z)$ where the envelope $\tilde \Psi$ is slowly varying on axial lengthscales $1/k$. Substituting this ansatz into \eref{eq:Helmholtz}, the contribution $\partial_z^2 \tilde \Psi$ may then be neglected against higher orders in $k$, yielding the paraxial Helmholtz equation for the evolution of  $\tilde \Psi$ \cite{PaganinXRay}:
\begin{equation}
  \left(2 \I k \partial_z + \Delta_\perp - 2k^2 (\delta - \I \beta)   \right) \tilde \Psi=0. \label{eq:ParHelmholtz}
\end{equation}
Here, $\Delta_\perp$ denotes the Laplacian in the lateral coordinate $\bi x$ and quadratic terms in $\delta, \beta$ have been neglected as these decrements of the refractive index are typically several orders of magnitude smaller than one in the considered X-ray regime.

The projection approximation corresponds to a description of the scattering interaction by geometrical optics: within the sample domain $z \in [-L;0]$ in \Fref{figure1} containing the support of $\delta, \beta$, diffraction of traversing X-rays is neglected by omitting the lateral coupling term $\Delta_\perp \tilde \Psi$ in \eref{eq:ParHelmholtz}. Although this approximation may seem crude for visible light, it is typically very accurate for the large wavenumbers $k$ and comparably thin objects encountered in X-ray radiography. Under the assumption $\Delta_\perp \tilde \Psi = 0$, \eref{eq:ParHelmholtz} reduces to the ordinary differential equation $\left( \partial_z  + \I k (\delta - \I \beta)   \right) \tilde \Psi=0$. Solving this equation for $z \in [-L;0]$, we obtain the following approximation for the scattered wave field $\Psi_z := \Psi(\cdot, z) = \E^{\I k z} \tilde \Psi(\cdot,z)$ at the exit-surface $z = 0$:
 \begin{equation}
  \Psi_0(\bi x )   = \underbrace{\tilde{\Psi}(\bi x, -L )}_{=: P(\bi x)} \cdot  \underbrace{\exp\left(  -  \I k  \int_{-L}^0 (\delta(\bi x ,z) - \I \beta(\bi x ,z)) \; \rmd z  \right)}_{=: O(\bi x)}  \label{eq:1}.
 \end{equation}
 $P $ denotes the \emph{probe function} describing the incident illumination wave field and $O$ is the \emph{object transmission function} encoding the sample structure \cite{Thibault2009ProbeObjectFct}. The exponential of the line integrals over $\delta$ and $\beta$ describes an accumulation of phase shifts and attenuation, respectively, of the X-rays traversing the object.
 
 Between the sample's exit surface $E_0$ and the detector plane $E_d$ in \Fref{figure1}, the scattered wave field propagates in vacuum, characterized by a constant refractive index $n=1$. In this stage of the considered imaging system, diffractive effects are retained in our model by solving \eref{eq:ParHelmholtz} analytically for $\delta = \beta = 0$. This yields a relation between the wave field $\Psi_d$ at the detector and $\Psi_0$ described by the \emph{Fresnel propagator} $\DF$ \cite{PaganinXRay}:
  \begin{eqnarray}
   \Psi_d(\bi x ) =  \DF(\Psi_0)(\bi x &) := &\E^{\I (k d - m \pi /4) }  \left( \frac{ k  }{ 2 \pi d } \right)^{m/2}  \exp \left( \frac{\I k \bi x ^2}{2 d} \right) \nonumber \\ 
   &\bi\cdot \int_{\mR^m} &\Psi_0 (\bi{x'} )  \exp \left( \frac{\I k \bi{x'} ^2}{2 d}  \right) \exp \left( -\frac{\I k \bi x \bi  \cdot \bi{x'} }{ d}    \right)\; \rmd \bi{x'} .  \label{eq:2}
 \end{eqnarray}
  
Physically, the detection of the scattered radiation in the plane $E_d$ is limited to recording wave \emph{intensities} represented by the squared modulus of the propagated wave field $\Psi_d$. This means that the phase information is lost in the measurement process, a defect known as the phase problem of optics. Combining \eref{eq:1} and \eref{eq:NonDimFresnel}, we find that the observed data in our model of propagation-based phase contrast imaging is given by
\begin{equation}
 I_d = |\Psi_d|^2 =  |  \DF ( P \cdot O) |^2. \label{eq:Intensities}
\end{equation}
The principal imaging problem considered in this work lies in reconstructing the object transmission function $O$ from intensities $I_d$ by solving the \emph{phase retrieval problem} \eref{eq:Intensities}. We show that \thmref{thm:1} guarantees uniqueness of such a reconstruction if the probe function $P$ is known and of suitable form.}

To this end, we introduce dimensionless coordinates $\bxi := (k/d)^{1 / 2} \bi{x}$, $\bxi' := (k/d)^{1 / 2}\bi{x}'$ and corresponding fields $\wF(\bxi)  := \exp( \I  \bxi^2 /2)$, $\psi_z(\bxi) := \Psi_z(\bi x)$ in \eref{eq:2}. This yields
 \begin{equation}
  \psi_d = \tDF( \psi_0 ) := \gamma \E^{\I k d} \wF \cdot \FF( \wF \cdot \psi_0) \label{eq:NonDimFresnel}
 \end{equation}
   where $\FF$ denotes the $m$-dimensional Fourier transform and $\gamma := \E^{-\I m \pi /4  } $. Defining $I(\bxi) := I_d(\bi x)$, $p(\bxi) := P(\bi x)$, $o(\bxi) := p(\bi x)$, $h:= p\cdot ( o -1 )$ and using \eref{eq:NonDimFresnel}, \eref{eq:Intensities} becomes
   \begin{equation}
 I = |  \tDF ( p \cdot o) |^2 = |  \tDF ( p ) +  \tDF ( h ) |^2  =    \left|  \frac{\E^{-\I k d}}{\gamma \wF} \tDF ( p ) + \FF( \wF \cdot h ) \right|^2. \label{eq:IntensitiesNoDim}
\end{equation}
Physically, $h$ describes the perturbation of the X-ray wave field induced by the scatterer. {By our assumption of compactly supported $\delta$ and $\beta$, \eref{eq:1} implies that $o(\bxi) = 1$ for all $\bxi \in \mR^m \setminus \Omega$ outside a bounded domain $\Omega \subset \mR^m$ so that $h$ has compact support.}
Moreover, $O$ can be recovered uniquely from $h$ if $p$ is known and everywhere nonzero. Hence, feasibility of the considered imaging problem reduces to the question whether any compactly supported $h$ is uniquely determined by data of the form \eref{eq:IntensitiesNoDim}.
 	\begin{figure}[hbt!]
 	 \centering
 	 \includegraphics[width=.24\textwidth]{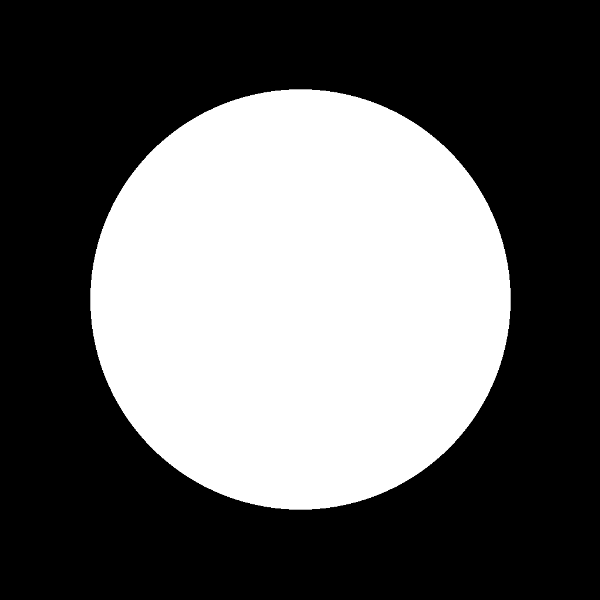}\hfill \includegraphics[width=.24\textwidth]{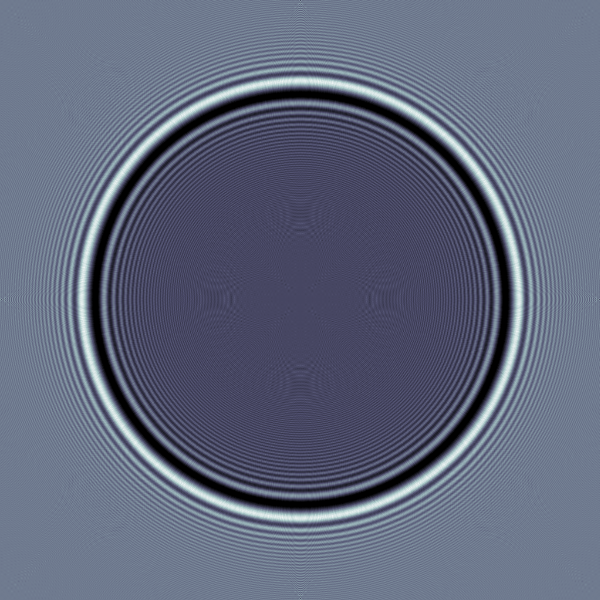} \hfill \includegraphics[width=.24\textwidth]{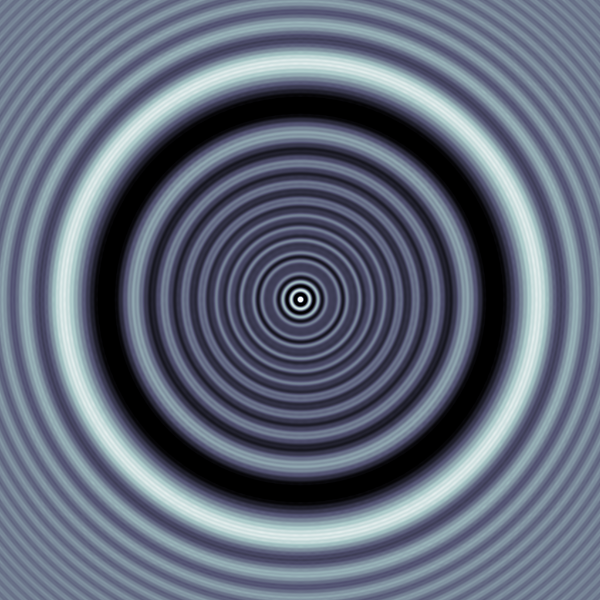} \hfill \includegraphics[width=.24\textwidth]{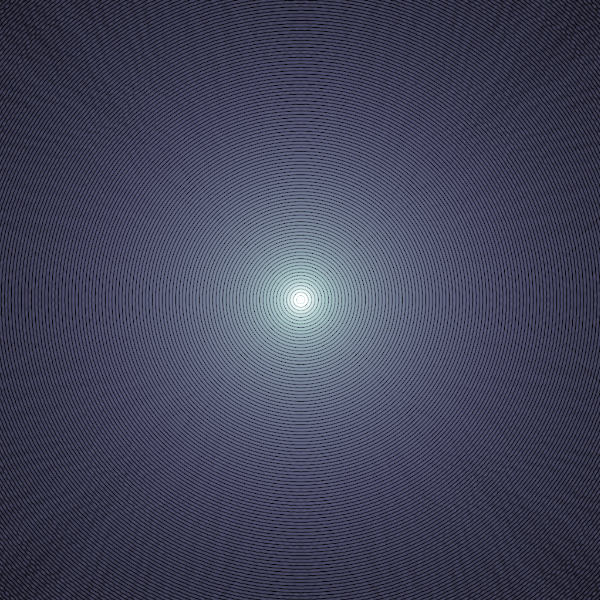}
 	 \caption{Simulated near- and far-field phase contrast for a disc-shaped object transmission function $O$ of radius $R$ (left figure, $O = \exp(-\I - \frac 1 {10} )$ inside the disc), corresponding to scattering on a refracting and weakly absorbing specimen. Center figures (linear color scale) show intensity data simulated according to \eref{eq:Intensities} for $P =1$ and $2\pi d/(kR^2) = 10^{-4}$  and $10^{-3}$, respectively, representing the near-field imaging regime considered in this work. The rightmost plot (log-scale) shows the frequently studied case of Fourier intensities $| \FF(P \cdot O) |^2$, valid in the far-field limit $d \to \infty$. \label{FarfieldNearfieldDemo}}
 	\end{figure}

As discussed in the introduction, a frequently studied phase retrieval problem is the reconstruction of a compactly supported signal $f$ from the squared modulus of its Fourier transform $|\FF(f)|^2$. Note that the intensity data defined in \eref{eq:IntensitiesNoDim} is quite different from this well-known setting in that the object dependent part $\FF( \wF \cdot h )$ is superimposed with the  \emph{reference signal} $\frac{\E^{-\I k d}}{\gamma \wF} \tDF ( p )$ which is essentially given by the propagated probe. {It is the presence of this reference term in \eref{eq:IntensitiesNoDim} which will be shown to permit application of \thmref{thm:1} for suitable (non-compactly supported) fields $p$ and thus yield uniqueness of the considered imaging problem. On the contrary, the present work's approach is inapplicable to the classical phase retrieval problem, studied for instance in \cite{Klibanov1986FresnelUnique,Klibanov20062DUniquePurePhase,Klibanov20141Dunique,Klibanov1995}, where the \emph{total} signal $f \approx p\cdot o$ (and not just the perturbation $h$ induced by the scatterer) is assumed to be compactly supported.}

Phase retrieval from plain Fourier data, however, arises naturally as the far-field limit of \eref{eq:IntensitiesNoDim}: if the exit wave $\Psi_0 = P\cdot O$ is non-negligible only within a domain of diameter $\ll (d/k)^{\frac 1 2}$, then the approximation $\wF \approx 1$ is justified so that $|  \tDF ( p \cdot o) |^2 \approx  |  \FF( p \cdot o ) |^2$
by \eref{eq:NonDimFresnel}.
The qualitative differences between near- and far-field imaging are illustrated by the numerical examples plotted in \Fref{FarfieldNearfieldDemo}. Due to the moderately small detector distances in the near-field setting, the measured intensities typically show wavy fringes arising from diffractive propagation effects along with traces of direct contrast, i.e.\ real-space representations of bulk object features such as the disc-shaped shadow in the mid-left and mid-right images.
On the contrary, far-field data represented by the rightmost image in \Fref{FarfieldNearfieldDemo} is by definition given by Fourier space information only, corresponding to an entirely diffractive encoding of the sample structure.
%

Mathematically, the difference between near- and far-field governed by the Fresnel propagator $\tDF$ and the Fourier transform $\FF$, respectively, is perhaps most prominent for illumination by plane waves, described by a constant probe field $p   = p_0 \in \mC \setminus \{0\}$: while $\FF$ maps constants to scaled Dirac-deltas, meaning that the propagated probe merely provides a strongly localized reference wave in the far-field, a laterally uniform background signal is obtained in the Fresnel regime since $ \E^{-\I k d} \tDF(p_0) = p_0 $. By \eref{eq:IntensitiesNoDim}, the near-field intensity data under plane wave illumination is thus of the form
\begin{equation}
I(\bxi)  =  \left| \frac{ p_0 }{ \gamma } \exp\left( - \frac{\I  \bxi^2}  2 \right) +  \FF(\wF \cdot h)(\bxi) \right|^2. \label{eq:4} 
\end{equation}
The forward map $F: h \mapsto I$ defined by \eref{eq:4} matches the assumptions of \thmref{thm:1} with $\check p =(2\pi)^{m/2} (p_0 / \gamma) \delta_0$ and $\alpha = - \I / 2$ where $\delta_0$ denotes the Dirac-delta distribution centered at $0$. Linearizing $F$ at $h = 0$ yields an operator $F_{\Text{lin}}$ of the form \eref{eq:defFlin}, representing the physically relevant limit of weakly scattering samples.

A more realistic class of probe functions is given by Gaussian beams propagating in axial direction. These constitute analytical solutions to the paraxial Helmholtz equation \eref{eq:ParHelmholtz} in vacuum, such that the lateral intensity profile is everywhere of Gaussian shape \cite{Teich1991Photonics}. More precisely, the propagated probe beam in the detector plane $E_d$ is of the form
\begin{equation}
 \tDF(p) (\bxi) = p_0 \E^{ \I k d} \exp(  \alpha_0  \bxi^2 )\quad \text{with} \quad p_0 \in \mC \setminus \{0\}, \,\Re(\alpha_0) < 0  
\end{equation}
If the focal point of the beam is located left of $E_d$, i.e.\  if the wave field is divergent at the detector, we further have $\Im(\alpha_0) \leq 0$ so that the resulting probe term
\begin{equation}
\frac{ \E^{-\I k d}}{ \gamma \wF} \tDF(p) (\bxi) = \frac{ p_0 }{ \gamma } \exp(   \alpha \bxi^2 )\quad \text{with} \quad \alpha = \alpha_0 - \frac{\I}{2} \label{eq:5}
\end{equation}
to be substituted into \eref{eq:IntensitiesNoDim} is in accordance with the assumptions of \thmref{thm:1}.

{All in all, our analysis shows that \thmref{thm:1} is applicable to the considered setting of near-field phase contrast imaging, i.e.\  we have proven the following corollary:
\vspace{1em}
\begin{cor}[Uniqueness of near-field phase contrast imaging] \label{cor:prob1}
 Any object transmission function $O \in  1+ \Sdashc{m}$, arising from a compactly supported sample illuminated by a known probe function $P$ of Gaussian- or plane wave shape via \eref{eq:1}, is uniquely determined 
 by intensity data $(I_{d})_{|U}$ of the form \eref{eq:Intensities} measured at a single distance $d > 0$ on any open set $U \subset \mR^m$.
 Moreover, $O$ is uniquely determined by linearized intensities $(I_{d, \Text{lin}})_{|U}$ with $I_{d, \Text{lin}} := I_{d} - \left|  \DF ( P \cdot (O-1))  \right|^2$, i.e.\ uniqueness is retained in the weak scattering limit.
\end{cor}
\begin{proof}
 Apply \thmref{thm:1} to the operators defined by \eref{eq:IntensitiesNoDim}, \eref{eq:4}, \eref{eq:5}. Use $o = 1 + h / p $.
\end{proof}
\vspace{1em}
Note that we only claim unique reconstruction of $O$ from the data, not of the line integrals over $\delta$ and $\beta$  in \eref{eq:1}. Indeed, recovery of the the latter requires inversion of a pointwise exponential which is potentially non-unique due to the $2\pi$-periodicity of $\exp$ in the imaginary part. Physically, this gives rise to the \emph{phase-wrapping} problem: refractive phase shifts of the transmitted radiation by more than a wavelength may not be measured unambiguously. For many scatterers of interest, however, the induced refraction is sufficiently weak for phase-wrapping ambiguities to be excluded a priori.

On the other hand, the line integrals in \eref{eq:1} provide merely a shadow image of the actual sample. In order to resolve the spatially varying refractive index $n = 1-\delta + \I \beta$ itself, the object in \Fref{figure1} can be rotated by angles $-\theta_\perp \in \mSn{1}$ in the plane spanned by the $x_1$- and $z$-axes, changing the incident angle of the probing X-rays. This is the setting of \emph{phase contrast tomography}. Mathematically, it corresponds to a composition of the axial integration in \eref{eq:1} with rotations of $\delta, \beta$, which is equivalent to integrating along the axes $\theta_\perp$. Identifying the rotational plane with $\mR^2$, the ensemble of these rotated line integrals may be written as the two-dimensional Radon transform $\RR_2$ \cite{Natterer}
		\begin{equation}
		 \RR_2 f(\theta, x) = \int_{\mR}  f(x \theta + y \theta_\perp) \; \D y \quad \text{for} \quad f: \mR^2 \to \mC, x \in \mR, \, \theta_\perp \perp \theta \in \mSn 1 \label{eq:Radon}
		\end{equation}		
Setting $\RR g(\theta, x_1, \ldots , x_m):= \RR_2(g(\cdot,x_2,\ldots,x_m,\cdot))(\theta, x_1)$ for $g:\mR^{m+1}\to \mC$, the object transmission functions $O_\theta$ for the different incident angles are thus given by
 \begin{equation}
  O_\theta = \exp \left( - \I k \RR_\theta(\delta - \I \beta) \right)  \quad \text{where} \quad \RR_{\theta}f := \RR(f)(\theta, \cdot). \label{eq:ORadon}
 \end{equation}
 According to \cref{cor:prob1}, these are uniquely determined by the corresponding near-field intensities $\{I_{d,\theta}\}_{\theta \in \mSn{1}}$. By \eref{eq:ORadon}, this implies that the object $\delta - \I \beta$ can be recovered from the tomographic data if both the pointwise exponential and the Radon transform in \eref{eq:ORadon} are invertible. The former requires to restrict to non-phase-wrapping objects which induce refractive phase shifts of at most one wavelength. These considerations lead to the following corollary which is proven in \sref{S4}:
 \vspace{1em}
\begin{cor}[Uniqueness of phase contrast tomography modulo phase-wrapping]\label{cor:prob2}
 Any compactly supported object $\delta - \I \beta \in \Sdashc{m+1}$ s.t. $0 \leq k \RR(\delta) < 2 \pi$, illuminated by a known Gaussian beam- or plane wave probe function $P$, is uniquely determined by tomographic intensity data $\{(I_{d,\theta})_{|U}\}_{\theta \in V}$ of the form \eref{eq:Intensities}, \eref{eq:ORadon} measured at a single distance $d > 0$ on any open sets $U \subset \mR^m, V \subset \mSn{1}$. Moreover, uniqueness is retained in the weak scattering limit represented by a linearization of \eref{eq:Intensities}, \eref{eq:ORadon} at $\delta - \I \beta = 0$.
\end{cor}
\vspace{1em}

}

We conclude this section with some remarks concerning possible generalizations of the considered idealized setup. Firstly, note that the setting of \thmref{thm:1} allows for much more general probe functions than the above examples, although this generality is difficult to translate into a particular set of admissible choices.
Moreover, if the illumination is unknown, an additional flat field measurement of the intensity profile without an object in the beam line may be used to normalize the scattering intensities  by division through the empty beam data. This procedure yields a good approximation of the hypothetical data in the idealized case of plane wave illumination if the probe field varies only on coarse length scales compared to the object structure \cite{Hagemann2014EmptyBeam}. Finally, note that the above derivations assume a parallel beam geometry, i.e.\  non-divergent incident radiation, whereas state-of-the-art phase contrast imaging setups often use cone beams emanating from a point-like micro-focus X-ray source. However, we emphasize that measurements obtained with cone beam illumination can be associated with an approximate parallel beam setup via the Fresnel scaling theorem \cite{Papoulis1968FresnelScale,Pogany1997noninterferometric} or incorporated explicitly into the model \cite{Louis2013Cone}.

\section{Mathematical preliminaries} \label{S3}

In the following, we review elements of the theory of entire functions. The given overview is based on the more detailed treatment in \cite{BoasEntire,ConwayOneComplex1,FreilingEntire}. The relevance of entire functions to phase retrieval is due to the well-known Paley-Wiener-Schwartz theorem:
	\vspace{1em}\begin{theorem}[Paley-Wiener-Schwartz \cite{Hoermander}]\label{thm:PaleyW}
		Let $K\subset \mR^m$ compact and convex. Then, if $u$ is a distribution of order $N \in \mN \cup \{ 0\}$ with support contained in $K$ and $\hat u := \FF(u)$, $\hat u$ has an extension to an entire function and there exists a constant $C> 0$ such that
		\begin{equation}
		 |\hat u(\bxi)| \leq C ( 1+ |\bxi|)^N \exp\left( \sup_{x \in K} \Im(\bxi ) \cdot x \right) \MTEXT{for all} \bxi \in \mC^m \label{eq:PaleyWBound}
		\end{equation}
		Conversely, any entire function $\hat u$ satisfying \eref{eq:PaleyWBound} is the complex extension of the Fourier transform of a distribution $u$ of order $\leq N$ and support in $K$.
	\end{theorem}\vspace{1em}
\noindent The essence of the result is that compactly supported objects can be identified with entire functions of limited growth via the Fourier transform. Owing to the relatedness of the transforms by \eref{eq:NonDimFresnel}, the same holds true for the Fresnel propagator.

For simplicity, we restrict to univariate entire functions $f: \mC \to \mC$. In order to apply the presented theory in the proof of  \thmref{thm:1}, multivariate functions $g: \mC^m \to \mC$ will be identified with families $\{g_{\bxi_0}\}_{\bxi_0 \in \mC^{m-1}}$ with $g_{\bxi_0}(\xi):=g(\xi,\bxi_0)$. The principal idea lies in estimating the growth behavior of entire functions, characterized by
        \begin{equation}
                M_f(r) := \max_{\xi \in \mC: |\xi| = r} |f(\xi)| \qquad \text{and} \qquad m_f(r) := \min_{\xi \in \mC: |\xi| = r} |f(\xi)|. \label{eq:DefmM}
        \end{equation}
Asymptotic bounds on $f$ give rise to the definition of its \emph{order} $\lambda_f$ and \emph{type} $\tau_f$:
\numparts
\begin{eqnarray}
 \lambda_f &:= \cases{0&for $f$ constant \\
\limsup_{r \to \infty} \frac{ \log \log M_f(r) }{ \log r } &else \\} \label{eq:DefOrder} \\
\tau_f &:= \cases{0&if $\lambda_f =0$ \\
\limsup_{r \to \infty} r^{-\lambda_f} \log M_f(r) &else \\}
\end{eqnarray}
\endnumparts
We say that order 1 entire functions are of exponential order. According to \thmref{thm:PaleyW}, the Fourier transform of any compactly supported function is an entire function of at most exponential order of finite type.
 Moreover, note that order and type of an entire function $f$ is preserved under \emph{Schwarz reflection} $f \mapsto \adj f$, defined by
  \begin{equation}
   \adj{f} (\xi ) := \conj{f (\conj{\xi})} \MTEXT{for all} \xi \in \mC.
  \end{equation}

By definition, adding a lower order entire function $g$, $\lambda_g < \lambda_f$ to $f$ does not change its order, nor its type. Likewise, it is clear that multiplication with $g$ cannot \emph{increase} any of these properties. For the growth estimates to be made, we further need that they may neither \emph{decrease} if $g$ is of at most exponential order and not identically zero:
	\vspace{1em}\begin{lemma}[Decay bounds for low order entire functions \cite{BoasEntire}]\label{lem:MaxDecayExpOrder}
		Let $f$ be an entire function of order $0 \leq \lambda_f \leq 1$ that is not identically zero and let $\varepsilon > 0$. Then
		\begin{equation*}
		 \limsup_{r \to \infty} m_f(r) M_f(r)^{1+\varepsilon} > 0
		\end{equation*}
		In particular, if $f$ is at most of exponential type $\tau_f$, then $\limsup_{r \to \infty} m_f(r) \E^{(\tau_f+\varepsilon) r} = \infty$.
	\end{lemma}\vspace{1em}
\noindent The essential message of \lemref{lem:MaxDecayExpOrder} is that non-vanishing factors of at most exponential order may never weaken super-exponential growth.

From the perspective of the theory of entire functions, the images of the operators $F, F_{\text{lin}}$ considered in \thmref{thm:1} are of a very particular structure: by expanding the squared moduli in  \eref{eq:defF} and \eref{eq:defFlin}, we obtain linear combinations of Fourier transforms of compactly supported distributions, i.e.\ order $\leq\,$1 entire functions according to \thmref{thm:PaleyW}, scaled with different quadratic-exponential factors $\exp( \alpha_j \bxi^2 )$. The uniqueness statement in \thmref{thm:1} will follow from the surprising insight that these summands may never balance one another due to their inconsistent super-exponential growth behavior in the complex plane, i.e.\ are linearly independent:
	\vspace{1em}\begin{lemma}[Linear independence of Fresnel factors]\label{lem:LinearIndependence}
		For $K \in \mN$, let $f_1, \ldots, f_K: \mC \to \mC$ be entire functions of order $ \leq 1$  such that for pairwise different  $\alpha_1, \alpha_2, \ldots, \alpha_K \in \mC$
		\begin{equation}
		 \sum_{j = 1}^N f_j(\xi)  \exp\left( \alpha_j \xi^2 \right)  = 0 \MTEXT{for all} \xi \in \mC \label{eq:LinComb}.
		\end{equation}
		Then $f_1 = f_2 = \ldots = f_N = 0$.
	\end{lemma}
	\begin{proof}
	 We choose $j_{\max} \in \{ 1, \ldots, K\}$ s.t. $| \alpha_{j_{\max}} | = \max_{1 \leq j \leq K } |\alpha_j |$ and show  $f_{j_{\max}} = 0$. The general statement then follows by inductively repeating this maximum index choice.
	 
	 Writing $\alpha_{j} = |\alpha_{j}| \exp( \I \varphi_j )$, we consider a diagonal in the complex plane given by
	 \begin{equation*}
	  D := \left\{ r \exp\left( -\I \varphi_{j_{\max}}/2 \right) : r \in \mR \right\}.
	 \end{equation*}
	 Then we have by construction
	 \begin{equation*}
	  \left|\exp( \alpha_{j} \xi^2 )\right| = \exp\left( |\alpha_{j}| \cos( \varphi_j - \varphi_{j_{\max}} ) |\xi|^2 \right) \MTEXT{for all}  \xi \in D.
	 \end{equation*}
	 Since for all $ j \in \{1, \ldots , K \} \setminus \{ j_{\max} \}$ either $|\alpha_{j}| < | \alpha_{j_{\max}} |$ or $\varphi_j \notin \varphi_{j_{\max}} + 2\pi \mZ$ and $|\alpha_{j}| = | \alpha_{j_{\max}} |$ holds true, there exists an $\varepsilon > 0$ such that for all $\xi \in D, j \neq j _{\max}$
	 \begin{equation*}
	  \left|\exp( \alpha_{j} \xi^2 )\right| \leq \exp\left( (| \alpha_{j_{\max}} | - 2 \varepsilon) |\xi|^2 \right) = \exp\left( -2 \varepsilon |\xi|^2 \right) \left|\exp( \alpha_{j_{\max}} \xi^2 )\right|,
	 \end{equation*}
	 i.e.\ the growth of $\exp( \alpha_{j_{\max}} \xi^2 )$ along $D$ exceeds that of all other $\exp(\alpha_j \xi^2)$ by a super-exponential factor. Using \eref{eq:LinComb} and the triangle inequality, this implies for all $\xi \in D$
	 \begin{eqnarray}
	      \left|f_{j_{\max}}(\xi)\right| \exp\left(  \varepsilon |\xi|^2  \right) &\leq& \sum_{j = 1, \,j \neq j_{\max}}^N     \frac{ \exp\left(   \varepsilon |\xi|^2 \right)  \left| f_j(\xi) \right| \left| \exp( \alpha_j \xi^2 ) \right| }{  \left| \exp( \alpha_{j_{\max}} \xi^2 )\right| }  \nonumber \\ 
	      &\leq&   \sum_{j = 1, \,j \neq j_{\max}}^N \left| f_j(\xi) \right| \exp\left(  - \varepsilon |\xi|^2 \right). \label{eq:GrowthEst}
	 \end{eqnarray}
	 
	 As the $f_j$ are of at most exponential order, the right hand side of \eref{eq:GrowthEst} must vanish for $|\xi| \to \infty$. Using the notation introduced in \eref{eq:DefmM} and the definition of the order in \eref{eq:DefOrder}, we obtain for the same reason
	 \begin{eqnarray}
	  \limsup_{r \to \infty} \frac{ \log \log \left( M_{f_{j_{\max}}}(r)^{1+\varepsilon} \right) }{ \log r } &=& \limsup_{r \to \infty} \frac{ \log \log   M_{f_{j_{\max}}}(r)  }{ \log r }  \leq 1  \nonumber \\
	  &<& \limsup_{r \to \infty} \frac{ \log \log  \left( \exp\left(   \varepsilon r^2  \right) \right) }{ \log r }.  
	 \end{eqnarray}
	 Hence, there exists a constant $C > 0$ such that
	 \begin{equation}
	  C M_{f_{j_{\max}}}(r)^{1+\varepsilon} \leq  \exp\left(   \varepsilon r^2  \right) \MTEXT{for all} r \in \mR. \label{eq:MfEst}
	 \end{equation}
         Now assume that $f_{j_{\max}} \neq 0$. Then an application of \lemref{lem:MaxDecayExpOrder}, \eref{eq:DefmM} and \eref{eq:MfEst} yields
         \begin{eqnarray}
	     \limsup_{\xi \in D, \, |\xi| \to \infty} |f_{j_{\max}}(\xi)| \exp\left(  \varepsilon |\xi|^2  \right) &\geq&  C \limsup_{ |\xi| \to \infty} m_{f_{j_{\max}}}(|\xi|) M_{f_{j_{\max}}}(|\xi|)^{1+\varepsilon} > 0 \nonumber \\
	     &=& \limsup_{  |\xi| \to \infty} \sum_{j = 1, \,j \neq j_{\max}}^N | f_j(\xi) | \exp\left(  - \varepsilon |\xi|^2 \right)  
         \end{eqnarray}
         in contradiction to \eref{eq:GrowthEst}. Hence, we conclude that $f_{j_{\max}} = 0$.
	\end{proof}
	\vspace{1em}

\section{Proof of the main results} \label{S4}

With the preparations of \sref{S3}, we are now in a position to prove \thmref{thm:1}. For completeness, we start by showing well-definedness:
\vspace{1em}\begin{lemma} \label{lem:WellDefined}
 The operators $F, F_{\Text{lin}}$ in \thmref{thm:1} are well-defined. If $\Re(\alpha) \leq 0$, then furthermore $F(\Sdashc{m}) \subset \Sdash{m}$ and $F_{\Text{lin}}(\Sdashc{m}) \subset \Sdash{m}$.
\end{lemma}\vspace{1em}
{\begin{proof}
 For $h \in \Sdashc{m}$, we have  $w \cdot h \in \Sdashc{m}$ so that $\FF(\check p)$ and $\FF(w \cdot h)$ have extensions to entire functions in $\mC^m$ by \thmref{thm:PaleyW}. In particular, this implies $\FF(\check p), \FF(w \cdot h) \in \CpR{\infty}m$. Hence, the same holds true for $F(h)$ and $F_{\Text{lin}}(h)$ defined by
 \begin{eqnarray*}
  F(h)(\bxi) &=& | \FF(\check p ) ( \bxi ) \exp(\alpha \bxi^2) + \FF(w \cdot h)(\bxi) |^2  \\
  F_{\Text{lin}}(h)(\bxi) &=& | \FF(\check p ) ( \bxi ) \exp(\alpha \bxi^2) + \FF(w \cdot h)(\bxi) |^2 - |\FF(w \cdot h)(\bxi)|^2.
 \end{eqnarray*}
 For $\Re(\alpha) \leq 0$, all terms inside the moduli are of at most algebraical growth in $\mR^m$ so that $F(h)$ and $ F_{\Text{lin}}(h) $ define tempered distributions, i.e.\  $F(h), F_{\Text{lin}}(h) \in \Sdash{m}$.
\end{proof}}\vspace{1em}
The proof of our central injectivity result in \thmref{thm:1} essentially amounts to rewriting the condition $F(h)_{|U} - F(\tilde h)_{|U} = 0$ for two objects $h , \tilde h$ that coincide in their images as a linear combination of the form \eref{eq:LinComb} and to apply \lemref{lem:LinearIndependence}:
\vspace{1em}\begin{proof}[Proof of \thmref{thm:1}]
	 Well-definedness has already been shown in \lemref{lem:WellDefined}.
	 To prove injectivity, let $h, \tilde h \in \Sdashc{m}$ be such that $F(h)_{|U} = F(\tilde h)_{|U}$ or $F_{\Text{lin}}( h)_{|U} = F_{\Text{lin}}(\tilde h)_{|U}$. By expanding the squared moduli in \eref{eq:defF} and \eref{eq:defFlin}, this yields for all $\bxi \in U$
	 \begin{eqnarray}
	 0 =&&\adj{\FF(\check p )} ( { \bxi } ) \cdot \exp(\conj{\alpha} \bxi^2) \cdot \FF(w \cdot (h - \tilde h ) ) \:\, ( \bxi )  \nonumber \\
	 &+&{\FF(\check p )} \:\, ( { \bxi } )   \cdot \exp( {\alpha} \bxi^2) \cdot \adj{\FF(w \cdot (h - \tilde h ) )}( \bxi ) \label{eq:ImageDiff} \\
	 &+s\cdot& \left(\adj{ \FF(w \cdot h) } ( {\bxi}) \cdot \FF(w \cdot h)(\bxi)  -\adj{ \FF(w \cdot \tilde h) } ( {\bxi}) \cdot \FF(w \cdot  \tilde h)(\bxi) \right). \nonumber 
	 \end{eqnarray}
	 where $f^\ast(\bxi) = \conj{f(\conj{\bxi})}$ and $s \in \{0,1\}$ depends on whether $F_{\Text{lin}}$ or $F$ is considered. Since the Schwarz reflection $^\ast$ preserves analyticity and $\check p, w \cdot h, w \cdot \tilde h \in \Sdashc{m}$ have compact support by assumption, the right hand side of \eref{eq:ImageDiff} defines an entire function in $\mC^m$ according to \thmref{thm:PaleyW}. By Taylor expansion in $U\subset \mR^m$, this implies in particular that \eref{eq:ImageDiff} holds for all $\bxi = (\xi_1, \ldots, \xi_m)\in \mC^m$ where the complex extension of $\bxi^2$  is defined as $\xi_1^2 + \ldots + \xi_m^2$ (not as a Hermitean inner product).
	 
	 In order to apply the 1D theory of section \ref{S3} to the considered $m$-dimensional setting, we define for fixed but arbitrary $\bxi_0 \in \mR^{m-1}$
 	 \begin{equation} \label{eq:def_a_b}
	  \eqalign{a_{\bxi_0}(\xi ) := \FF(w \cdot h )( \xi, \bxi_0 )  \\
	   \tilde a_{\bxi_0}(\xi ) := \FF(w \cdot \tilde h )( \xi, \bxi_0 )  \\
	  b_{\bxi_0}(\xi ) := \FF(\check p ) (  \xi , \bxi_0 ) \cdot \exp(\alpha \bxi_0^2).  }
	 \end{equation}
	 Substituting \eref{eq:def_a_b} into \eref{eq:ImageDiff} yields for all $\xi \in \mC$
	 \begin{eqnarray}
	  0 =&& \adj{b_{\bxi_0}}(\xi) \cdot ( a_{\bxi_0}(\xi) - \tilde a_{\bxi_0}(\xi) ) \cdot \exp(\conj \alpha \xi^2) \nonumber \\
	 &+& {b_{\bxi_0}}(\xi) \cdot ( \adj{a_{\bxi_0}}(\xi) - \adj{\tilde a_{\bxi_0}}(\xi) ) \cdot \exp(  \alpha \xi^2) \label{eq:a_b_lincomb} \\
	 &+s\cdot&\left( \adj{a_{\bxi_0}}(\xi) \cdot  a_{\bxi_0} (\xi) - \adj{\tilde a_{\bxi_0}}(\xi) \cdot  \tilde a_{\bxi_0} (\xi) \right). \nonumber 
	 \end{eqnarray}
	 Since $a_{\bxi_0}, \tilde a_{\bxi_0}, b_{\bxi_0}: \mC \to \mC$ are entire functions of at most exponential order by \thmref{thm:PaleyW} and $\alpha \notin \mR$ holds by assumption, \eref{eq:a_b_lincomb} exactly matches the setting of \lemref{lem:LinearIndependence} with 
	  	 \begin{equation} \label{eq:def_f1_f2_f3}
	  \eqalign{f_1 := \adj{b_{\bxi_0}}\cdot ( a_{\bxi_0} - \tilde a_{\bxi_0} )  \\
		   f_2 := {b_{\bxi_0}}\cdot ( \adj{a_{\bxi_0}} - \adj{\tilde a_{\bxi_0}} )  \\
	           f_3 := s \cdot(  \adj{a_{\bxi_0}} \cdot  a_{\bxi_0} - \adj{\tilde a_{\bxi_0}} \cdot  \tilde a_{\bxi_0} ).  }
	 \end{equation}
	 and exponents $\alpha_1 := \conj \alpha $, $\alpha_2 := \alpha$, $\alpha_3 = 0$. Hence, it follows that $f_1 = f_2 = f_3 = 0$.
	 
	 As $\bxi_0 \in \mR^m$ was arbitrary, the derived result holds for all $\bxi_0 \in \mR^m$. Re-substituting the expressions in $f_1 = \adj{b_{\bxi_0}}\cdot ( a_{\bxi_0} - \tilde a_{\bxi_0} )$ according to  \eref{eq:def_a_b}, this yields
	 \begin{equation} \label{eq:result}
	\FF(\check p )^\ast \cdot \FF(w \cdot (h - \tilde h ) )    = 0.
	 \end{equation}
	 Both factors in \eref{eq:result} define entire functions in $\mC^m$. As such, they are either almost everywhere nonzero in $\mR^m$ or vanish identically. By the assumption $\check p \in \Sdashc{m} \setminus \{ 0 \}$, the first case must hold for the factor $\FF(\check p )^\ast$. Consequently, \eref{eq:result} implies $\FF(w \cdot (h - \tilde h ) ) = 0$ and thus $h = \tilde h$ by bijectivity of $\FF: \Sdash m \to \Sdash m$ and the existence of $\frac 1 w $. By generality of $h, \tilde h \in \Sdashc{m}$, this proves injectivity of the operators
	 \begin{eqnarray*} 
	 F_U &:  \Sdashc m \to \mathscr{C}^\infty ( U );\; h \mapsto F(h)_{|U} \\
	 F_{\Text{lin},U} &:  \Sdashc m \to \mathscr{C}^\infty ( U );\; h \mapsto F_{\Text{lin}}(h)_{|U}. \;\;\;\;\;\;\;\;\;\;\;\;\;\;\;\;\;\;\;\;\;\;\;\;\;\;\; \;\;\;\;\;\;\;\;\;\;\;\;\;\;\;\;\;\;\,\qedhere
	 \end{eqnarray*} 
	\end{proof}\vspace{1em}

As outlined in section \ref{S2} and stated in \cref{cor:prob1}, the obtained result is applicable to a commonly-used model of near-field phase contrast imaging with X-rays. We conclude this section with the proof of \cref{cor:prob2}, extending the uniqueness statement to \emph{tomographic} imaging of compactly supported  non-phase-wrapping specimen:
\vspace{1em}

\begin{proof}[Proof of \cref{cor:prob2}]
 By \Cref{cor:prob1}, the object transmission functions $\{O_\theta\}_{\theta \in V}$ can be uniquely reconstructed from the data. As $\exp: \mC \to \mC$ is injective on $\{ z \in \mC: -2\pi < \Im(z) \leq 0\}$, the same holds true for $\{\RR_\theta(\delta - \I \beta)\}_{\theta \in V}$ due to the relations
 \begin{equation*}
  O_\theta = \exp ( - \I k \RR_\theta(\delta - \I \beta) ) \qquad \text{and} \qquad 0 \leq k \RR(\delta) < 2 \pi.
 \end{equation*}
 The Radon transform relates to the Fourier transform via the Fourier-Slice-Theorem. See \cite{Natterer} for details. In particular, $\{\RR_\theta(\delta - \I \beta)\}_{\theta \in V}$ uniquely determines $\FF(\delta - \I \beta)$ on some wedge-shaped open subset $W_V \subset \mR^{m+1}$ which is defined by the angles in $V$. On the other hand, $\FF(\delta - \I \beta)$ is entire analytic according to \thmref{thm:PaleyW} since $\delta - \I \beta$ is compactly supported. Hence, the values on $W_V$ uniquely determine $\FF(\delta - \I \beta)$ in $\mR^{m+1}$ and are thus sufficient for the reconstruction of the unknown object.
\end{proof}

\section{Conclusions} \label{S5}

In this paper, we have proven \thmref{thm:1} showing uniqueness of phase retrieval of compactly supported objects superimposed upon a certain class of known reference signals - both for the general nonlinear problem and for a linearization valid in the weak object limit. Notably, this uniqueness is deterministic and absolute - unlike many previous results for reconstructions from plain Fourier intensity data \cite{BruckSodin1979PhaseAmbiguity2D,Hayes1982Reducible,Barakat1984,Bates1984unique2DImage,Klibanov20062DUniquePurePhase,Klibanov20141Dunique}, which hold only modulo ``trivial'' ambiguities or sets of measure zero. 

As motivated in section \ref{S2}, the obtained result is applicable to a commonly used model of holographic near-field phase contrast imaging with coherent X-rays, an emerging technique of photonic nanoscopy. In the form of \cref{cor:prob2}, we have furthermore covered the setting of phase contrast tomography, permitting the resolution of three-dimensional variations of a sample's complex refractive index. It is often tacitly assumed \cite{Nugent2007TwoPlanesPhaseVortex,Burvall2011TwoPlanes} that at least two intensity measurements at different detector distances are necessary in these imaging setups to uniquely recover both the refractive phase shifts and the absorption imprinted upon the incident radiation traversing the specimen. The present work has shown this common belief to be untrue in principle. Moreover, the proven uniqueness theorem might also have some significance to \emph{electronic} imaging methods owing to the equivalence of the dynamics of paraxial electromagnetic waves applied herein, described by \eref{eq:ParHelmholtz}, and Schr\"odinger's equation in quantum mechanics.

However, in order to evaluate the practical applicability of our single-measurement result to experimental imaging setups with common point-like micro-focus X-ray sources, it needs to be supplemented with stability estimates. A promising starting point possibly lies in the analysis of the \emph{linear} forward operator arising in the weak scattering limit. By bounding the ill-posedness of the phase retrieval problem, stability results may shed a light onto how robust image reconstruction is against noisy measurements as well as with respect to systematic inaccuracies in the physical model - as induced for instance by the finite coherence and incomplete knowledge of the probing beam in realistic imaging setups.
Likewise, it is necessary to investigate whether the derived uniqueness is stable under relaxation of the paraxial- and projection approximations made in our description of the imaging system. This may elucidate how the present work relates to recent uniqueness results \cite{Klibanov2014PhaselessUnique1,Klibanov2014PhaselessUnique2} for phaseless inverse scattering within the framework of the full Helmholtz equation. Interestingly, the theorems derived therein require intensity data for a continuous interval of wavenumbers of the coherent incident waves. At least for the tomographic point-source-illumination setup considered in \cite{Klibanov2014PhaselessUnique1}, our single-measurement result - requiring illumination only with a single coherent probe - suggests that these assumptions might be relaxed considerably.

Recent numerical results for phase contrast tomography, obtained via alternating-projection-type algorithms \cite{Ruhlandt2014} and iteratively regularized Newton-type methods \cite{MyMaster} as proposed in \cite{Hohage2013}, indicate that reconstructions of general, refracting and absorbing samples are feasible, yet severely ill-posed. 
Hence, we conclude that considerable effort has to be put into tailoring reconstruction algorithms for single-distance phase contrast imaging in order to render it a viable option in realistic experimental situations.
The uniqueness results of the present work merely provide a first theoretical proof of concept, emphasizing the necessity and significance of further algorithmic, theoretical and experimental developments in this emerging imaging technique.

\ack
The author thanks Thorsten Hohage and Tim Salditt for the encouragement to dive into the fascinating theory of phase retrieval. Moreover, gratitude goes to Aike Ruhlandt for enlightening discussions concerning uniqueness and the neat sketch in \Fref{figure1}. The author furthermore thanks three anonymous referees for their inspiring comments, which have greatly helped to improve this article.
Support by the German Research Foundation DFG through the Collaborative Research Center 755 Nanoscale
Photonic Imaging is gratefully acknowledged.

\vspace{1em}
	
\section*{References}

\bibliographystyle{abbrvIP}
\bibliography{IP-100355_literature}

\begin{thebibliography}{10}

\bibitem{Akutowicz1956I}
Akutowicz E~J.
\newblock On the determination of the phase of a {F}ourier integral, i.
\newblock {\em Transactions of the American Mathematical Society}, pages
  179--192, 1956.

\bibitem{Akutowicz1957II}
Akutowicz E~J.
\newblock On the determination of the phase of a {F}ourier integral, ii.
\newblock {\em Proceedings of the American Mathematical Society},
  8(2):234--238, 1957.

\bibitem{Barakat1984}
Barakat R and Newsam G.
\newblock {Necessary conditions for a unique solution to two-dimensional phase
  recovery}.
\newblock {\em Journal of mathematical physics}, 25(11):3190--3193, 1984.

\bibitem{Bartels2012}
Bartels M, Priebe M, Wilke R~N, Kr{\"u}ger S~P, Giewekemeyer K, Kalbfleisch S,
  Olendrowitz C, Sprung M, and Salditt T.
\newblock {Low-dose three-dimensional hard {X}-ray imaging of bacterial cells}.
\newblock {\em Optical Nanoscopy}, 1(1):1--7, 2012.

\bibitem{Barty2008ceramicfoam}
Barty A, Marchesini S, Chapman H, Cui C, Howells M, Shapiro D, Minor A, Spence
  J, Weierstall U, Ilavsky J, et~al.
\newblock Three-dimensional coherent {X}-ray diffraction imaging of a ceramic
  nanofoam: Determination of structural deformation mechanisms.
\newblock {\em Physical review letters}, 101(5):055501, 2008.

\bibitem{Bates1984unique2DImage}
Bates R.
\newblock Uniqueness of solutions to two-dimensional {F}ourier phase problems
  for localized and positive images.
\newblock {\em Computer vision, graphics, and image processing},
  25(2):205--217, 1984.

\bibitem{BoasEntire}
Boas R~P.
\newblock {\em Entire functions}, volume~1.
\newblock Academic Press Inc., New York, 1954.

\bibitem{Bonse1965interferometric}
Bonse U and Hart M.
\newblock An {X}-ray interferometer.
\newblock {\em Applied Physics Letters}, 6(8):155--156, 1965.

\bibitem{BruckSodin1979PhaseAmbiguity2D}
Bruck Y~M and Sodin L.
\newblock On the ambiguity of the image reconstruction problem.
\newblock {\em Optics Communications}, 30(3):304--308, 1979.

\bibitem{Burvall2011TwoPlanes}
Burvall A, Lundstr{\"o}m U, Takman P~A, Larsson D~H, and Hertz H~M.
\newblock Phase retrieval in {X}-ray phase-contrast imaging suitable for
  tomography.
\newblock {\em Optics express}, 19(11):10359--10376, 2011.

\bibitem{Chapman2006}
Chapman H~N, Barty A, Marchesini S, Noy A, Hau-Riege S~P, Cui C, Howells M~R,
  Rosen R, He H, Spence J~C, et~al.
\newblock {High-resolution ab initio three-dimensional {X}-ray diffraction
  microscopy}.
\newblock {\em JOSA A}, 23(5):1179--1200, 2006.

\bibitem{Cloetens1999}
Cloetens P, Ludwig W, Baruchel J, {Van Dyck} D, {Van Landuyt} J, Guigay J, and
  Schlenker M.
\newblock {Holotomography: Quantitative phase tomography with micrometer
  resolution using hard synchrotron radiation {X}-rays}.
\newblock {\em Applied Physics Letters}, 75(19):2912--2914, 1999.

\bibitem{ConwayOneComplex1}
Conway J~B and Conway J~B.
\newblock {\em Functions of one complex variable}, volume~2.
\newblock Springer, 1973.

\bibitem{Fienup1986algorithm}
Fienup J and Wackerman C.
\newblock Phase-retrieval stagnation problems and solutions.
\newblock {\em JOSA A}, 3(11):1897--1907, 1986.

\bibitem{Fienup19782DPhaseRetrFeasible}
Fienup J~R.
\newblock Reconstruction of an object from the modulus of its {F}ourier
  transform.
\newblock {\em Optics letters}, 3(1):27--29, 1978.

\bibitem{FreilingEntire}
Freiling G and Yurko V.
\newblock {\em Introduction to the theory of entire functions}.
\newblock Schriftenreihe des Instituts f\"ur Mathematik der
  Universit\"at-Duisburg-Essen, Duisburg, 2003.

\bibitem{Guigay1977CTF}
Guigay J.
\newblock {F}ourier-transform analysis of {F}resnel diffraction patterns and
  in-line holograms.
\newblock {\em Optik}, 49(1):121--125, 1977.

\bibitem{Gureyev1999noninterferometric}
Gureyev T, Raven C, Snigirev A, Snigireva I, and Wilkins S.
\newblock Hard {X}-ray quantitative non-interferometric phase-contrast
  microscopy.
\newblock {\em Journal of Physics D: Applied Physics}, 32(5):563, 1999.

\bibitem{Hagemann2014EmptyBeam}
Hagemann J, Robisch A~L, Luke D, Homann C, Hohage T, Cloetens P, Suhonen H, and
  Salditt T.
\newblock Reconstruction of wave front and object for inline holography from a
  set of detection planes.
\newblock {\em Optics Express}, 22(10):11552--11569, 2014.

\bibitem{Hayes1982Reducible}
Hayes M~H and McClellan J~H.
\newblock Reducible polynomials in more than one variable.
\newblock {\em Proceedings of the IEEE}, 70(2):197--198, 1982.

\bibitem{Henke1993TypicalDeltaBeta}
Henke B~L, Gullikson E~M, and Davis J~C.
\newblock {X}-ray interactions: Photoabsorption, scattering, transmission, and
  reflection at {E=50-30\,000\,eV}, {Z=1-92}.
\newblock {\em Atomic data and nuclear data tables}, 54(2):181--342, 1993.

\bibitem{Hohage2013}
Hohage T and Werner F.
\newblock {Iteratively regularized {N}ewton-type methods for general data
  misfit functionals and applications to Poisson data}.
\newblock {\em Numerische Mathematik}, 123(4):745--779, 2013.

\bibitem{Hoermander}
H{\"o}rmander L.
\newblock {\em {The analysis of linear partial differential operators I}}.
\newblock Springer, Berlin, 2003.

\bibitem{Jonas2004TwoMeasUniquePhaseRetr}
Jonas P and Louis A.
\newblock Phase contrast tomography using holographic measurements.
\newblock {\em Inverse Problems}, 20(1):75, 2004.

\bibitem{Louis2013Cone}
Jonas P and Louis A.
\newblock Cone beam geometry for small objects in phase contrast tomography.
\newblock {\em Inverse Problems}, 29(9):095013, 2013.

\bibitem{Klibanov1986FresnelUnique}
Klibanov M.
\newblock Determination of a function with compact support from the absolute
  value of its {F}ourier transform, and an inverse scattering problem.
\newblock {\em Differential Equations}, 22(10):1232--1240, 1986.

\bibitem{Klibanov20062DUniquePurePhase}
Klibanov M~V.
\newblock On the recovery of a 2-d function from the modulus of its {F}ourier
  transform.
\newblock {\em Journal of mathematical analysis and applications},
  323(2):818--843, 2006.

\bibitem{Klibanov2014PhaselessUnique1}
Klibanov M~V.
\newblock Phaseless inverse scattering problems in three dimensions.
\newblock {\em SIAM Journal on Applied Mathematics}, 74(2):392--410, 2014.

\bibitem{Klibanov2014PhaselessUnique2}
Klibanov M~V.
\newblock Uniqueness of two phaseless non-overdetermined inverse acoustics
  problems in 3-d.
\newblock {\em Applicable Analysis}, 93(6):1135--1149, 2014.

\bibitem{Klibanov20141Dunique}
Klibanov M~V and Kamburg V~G.
\newblock Uniqueness of a one-dimensional phase retrieval problem.
\newblock {\em Inverse Problems}, 30(7):75004--75013, 2014.

\bibitem{Klibanov1995}
Klibanov M~V, Sacks P~E, and Tikhonravov A~V.
\newblock {The phase retrieval problem}.
\newblock {\em Inverse problems}, 11(1):1, 1995.

\bibitem{Krenkel2014BCAandCTF}
Krenkel M, T{\"o}pperwien M, Bartels M, Lingor P, Schild D, and Salditt T.
\newblock {X}-ray phase contrast tomography from whole organ down to single
  cells.
\newblock {\em SPIE Proceedings}, (9210):92120R, 2014.

\bibitem{Marchesini2003GoldSpFarfield}
Marchesini S, Chapman H, Hau-Riege S, London R, Szoke A, He H, Howells M,
  Padmore H, Rosen R, Spence J, et~al.
\newblock Coherent {X}-ray diffractive imaging: applications and limitations.
\newblock {\em Optics Express}, 11(19):2344--2353, 2003.

\bibitem{Marchesini2003PhysRevB}
Marchesini S, He H, Chapman H~N, Hau-Riege S~P, Noy A, Howells M~R, Weierstall
  U, and Spence J~C.
\newblock {X}-ray image reconstruction from a diffraction pattern alone.
\newblock {\em Physical Review B}, 68(14):140101, 2003.

\bibitem{MyMaster}
Maretzke S.
\newblock Regularized newton methods for simultaneous {R}adon inversion and
  phase retrieval in phase contrast tomography.
\newblock {\em arXiv preprint arXiv:1502.05073}, 2015.

\bibitem{Mayo2002quantitative}
Mayo S, Miller P, Wilkins S, Davis T, Gao D, Gureyev T, Paganin D, Parry D,
  Pogany A, and Stevenson A.
\newblock Quantitative {X}-ray projection microscopy: phase-contrast and
  multi-spectral imaging.
\newblock {\em Journal of microscopy}, 207(2):79--96, 2002.

\bibitem{MiaoNature1999CDIFirstExp}
Miao J, Charalambous P, Kirz J, and Sayre D.
\newblock Extending the methodology of {X}-ray crystallography to allow imaging
  of micrometre-sized non-crystalline specimens.
\newblock {\em Nature}, 400(6742):342--344, 1999.

\bibitem{Miao2003EColiFarfield}
Miao J, Hodgson K~O, Ishikawa T, Larabell C~A, LeGros M~A, and Nishino Y.
\newblock Imaging whole escherichia coli bacteria by using single-particle
  {X}-ray diffraction.
\newblock {\em Proceedings of the National Academy of Sciences},
  100(1):110--112, 2003.

\bibitem{Millane1990}
Millane R.
\newblock {Phase retrieval in crystallography and optics}.
\newblock {\em JOSA A}, 7(3):394--411, 1990.

\bibitem{Momose1995interferometric}
Momose A, Takeda T, and Itai Y.
\newblock Phase-contrast {X}-ray computed tomography for observing biological
  specimens and organic materials.
\newblock {\em Review of scientific instruments}, 66(2):1434--1436, 1995.

\bibitem{Natterer}
Natterer F.
\newblock {\em {The mathematics of computerized tomography}}, volume~32 of {\em
  {Classics of Applied Mathematics}}.
\newblock Society for Industrial and Applied Mathematics, 2001.

\bibitem{Nugent1996Propagationbased}
Nugent K, Gureyev T, Cookson D, Paganin D, and Barnea Z.
\newblock Quantitative phase imaging using hard {X}-rays.
\newblock {\em Physical review letters}, 77(14):2961, 1996.

\bibitem{Nugent2007TwoPlanesPhaseVortex}
Nugent K~A.
\newblock {X}-ray noninterferometric phase imaging: a unified picture.
\newblock {\em JOSA A}, 24(2):536--547, 2007.

\bibitem{Nugent2010coherent}
Nugent K~A.
\newblock Coherent methods in the {X}-ray sciences.
\newblock {\em Advances in Physics}, 59(1):1--99, 2010.

\bibitem{PaganinXRay}
Paganin D.
\newblock {\em {Coherent {X}-ray optics}}, volume~1.
\newblock Oxford University Press Oxford, 2006.

\bibitem{Paganin2002simultaneous}
Paganin D, Mayo S, Gureyev T~E, Miller P~R, and Wilkins S~W.
\newblock Simultaneous phase and amplitude extraction from a single defocused
  image of a homogeneous object.
\newblock {\em Journal of microscopy}, 206(1):33--40, 2002.

\bibitem{Paganin1998noninterferometric}
Paganin D and Nugent K~A.
\newblock Noninterferometric phase imaging with partially coherent light.
\newblock {\em Physical review letters}, 80(12):2586, 1998.

\bibitem{Papoulis1968FresnelScale}
Papoulis A.
\newblock Systems and transforms with applications in optics.
\newblock {\em McGraw-Hill Series in System Science, Malabar: Krieger, 1968},
  1, 1968.

\bibitem{Pogany1997noninterferometric}
Pogany A, Gao D, and Wilkins S.
\newblock Contrast and resolution in imaging with a microfocus {X}-ray source.
\newblock {\em Review of Scientific Instruments}, 68(7):2774--2782, 1997.

\bibitem{Teague1983TIE}
Reed~Teague M.
\newblock Deterministic phase retrieval: a {G}reen’s function solution.
\newblock {\em JOSA}, 73(11):1434--1441, 1983.

\bibitem{Ruhlandt2014}
Ruhlandt A, Krenkel M, Bartels M, and Salditt T.
\newblock Three-dimensional phase retrieval in propagation-based phase-contrast
  imaging.
\newblock {\em Physical Review A}, 89(3):033847, 2014.

\bibitem{Teich1991Photonics}
Teich M~C and Saleh B.
\newblock {\em Fundamentals of photonics}.
\newblock John Wiley \& Sons, New York, 1991.

\bibitem{Thibault2009ProbeObjectFct}
Thibault P, Dierolf M, Kewish C~M, Menzel A, Bunk O, and Pfeiffer F.
\newblock Contrast mechanisms in scanning transmission {X}-ray microscopy.
\newblock {\em Physical Review A}, 80(4):043813, 2009.

\bibitem{Thibault2006YeastCell}
Thibault P, Elser V, Jacobsen C, Shapiro D, and Sayre D.
\newblock Reconstruction of a yeast cell from {X}-ray diffraction data.
\newblock {\em Acta Crystallographica Section A: Foundations of
  Crystallography}, 62(4):248--261, 2006.

\bibitem{Turner2004FormulaWeakAbsSlowlyVarPhase}
Turner L, Dhal B, Hayes J, Mancuso A, Nugent K, Paterson D, Scholten R, Tran C,
  and Peele A.
\newblock {X}-ray phase imaging: Demonstration of extended conditions for
  homogeneous objects.
\newblock {\em Optics express}, 12(13):2960--2965, 2004.

\bibitem{Walther1963}
Walther A.
\newblock {The question of phase retrieval in optics}.
\newblock {\em Journal of Modern Optics}, 10(1):41--49, 1963.

\bibitem{Wilkins1996interferometric}
Wilkins S, Gureyev T, Gao D, Pogany A, and Stevenson A.
\newblock Phase-contrast imaging using polychromatic hard {X}-rays.
\newblock {\em Nature}, 384(6607):335--338, 1996.

\end{thebibliography}

\end{document}